\documentclass[a4paper,11pt]{amsart}
\usepackage[a4paper]{geometry}
\usepackage[utf8]{inputenc}
\usepackage[T1]{fontenc}
\usepackage{amssymb}
\usepackage[dvipsnames]{xcolor}
\usepackage{enumitem}
\usepackage{cancel}
\usepackage{lmodern}
\usepackage{mathtools}
\usepackage[]{microtype}
\UseMicrotypeSet[protrusion]{basicmath}
\usepackage{nicefrac}
\usepackage{stmaryrd}
\SetSymbolFont{stmry}{bold}{U}{stmry}{m}{n}
\usepackage{tabularx}
\usepackage{tikz}
\usepackage{verbatim}
\usepackage{esvect}
\usepackage{xspace}
\usepackage[initials, alphabetic]{amsrefs}
\usepackage[hidelinks]{hyperref}

\newtheorem{theoremintro}{Theorem}
\newtheorem{corollaryintro}[theoremintro]{Corollary}

\numberwithin{equation}{section}
\newtheorem{theorem}{Theorem}[section]
\newtheorem{corollary}[theorem]{Corollary}
\newtheorem{lemma}[theorem]{Lemma}

\newtheorem{proposition}[theorem]{Proposition}

\theoremstyle{definition}
\newtheorem{definition}[theorem]{Definition}
\newtheorem{remark}[theorem]{Remark}

\newtheorem*{definition*}{Definition}

\newtheorem*{claim*}{Claim}

\newcommand{\cG}{\mathcal{G}}

\newcommand{\cH}{\mathcal{H}}

\newcommand{\bbN}{\mathbb{N}}

\newcommand{\bbR}{\mathbb{R}}

\newcommand{\bbZ}{\mathbb{Z}}

\renewcommand{\phi}{\varphi}
\newcommand{\asdim}{\text{asdim}_B}
\newcommand{\abs}[1]{\left|#1\right|}
\newcommand{\degout}{\deg_{(X, \vv{R})}^{\rm{out}}}
\newcommand{\degoutz}{\deg_{(X, \vv{R}_0)}^{\rm{out}}}
\newcommand{\Degout}{\deg^{\rm{out}}}

\newcommand{\leaves}{\mathrm{leaves}}
\newcommand{\sch}{\mathrm{Sch}}

\title[Hyper-u-amenablity and Hyperfiniteness of Treeable Relations]{Hyper-u-amenablity and Hyperfiniteness of Treeable Equivalence Relations}
\date{}

\author{Petr Naryshkin}
\address{Petr Naryshkin,
Institute for Advanced Study, 1 Einstein Dr, Princeton, NJ 08540, US}
\email{penaryshkin@ias.edu}

\author{Andrea Vaccaro}
\address{Andrea Vaccaro, Institut Camille Jordan,
Université Claude Bernard Lyon 1, 
43 boulevard du 11 novembre 1918, 69622, Villeurbanne, France}
\email{vaccaro@math.univ-lyon1.fr}

\thanks{Both authors were funded by the Deutsche Forschungsgemeinschaft (DFG, German Research Foundation) under Germany’s Excellence Strategy – EXC 2044 – 390685587, Mathematics Münster – Dynamics – Geometry – Structure, the Deutsche Forschungsgemeinschaft (DFG, German Research Foundation) – Project-ID 427320536 – SFB 1442, and by the ERC Advanced Grant 834267 - AMAREC. The first-named author was also partially funded by the Dynasnet European Research Council Synergy project---grant number ERC-2018-SYG 810115, as well as the Marvin V. and Beverly J. Mielke Endowed Fund through the Institute of Advanced Study. 
}

\begin{document}
\begin{abstract}
We introduce the notions of u-amenability and hyper-u-amenability for countable Borel equivalence relations and we show that treeable, hyper-u-amenable countable Borel equivalence relations are hyperfinite. As corollaries of this result, we obtain that if a countable Borel equivalence relation is either:
\begin{enumerate}
\item measure-hyperfinite and equal to the orbit equivalence relation of a free continuous action of a virtually free group on a $\sigma$-compact Polish space,
\item treeable and equal to the orbit equivalence relation of a Borel action of an amenable group on a standard Borel space,
\item treeable, amenable and Borel bounded,
\end{enumerate}
then it is hyperfinite.
\end{abstract}
\maketitle

\tableofcontents

\section{Introduction}
One of the most relevant problems in the study of countable Borel equivalence relations on standard Borel spaces is clarifying the relationship between \emph{amenability}, meant in the sense of \cite{JKL} (see Definition \ref{def:amenability}), and \emph{hyperfiniteness}, which asserts that the equivalence relation is increasing union of Borel equivalence relations whose equivalence classes have finite cardinality. 
While it is well known that hyperfiniteness implies amenability (see \cite{SlamanSteel} and \cite{JKL}*{Proposition 2.13}), whether the converse holds remains a notoriously difficult open problem in the subject (\cite{JKL}*{\S 6.2.B}).

A related unsolved question is the one concerning the relation between hyperfiniteness and {measure-hyperfiniteness}. Recall that a countable Borel equivalence relation over a standard Borel space $X$ is \emph{measure-hyperfinite} if, for every Borel probability measure $\mu$ over $X$, there is some Borel $A \subseteq X$ with $\mu(A) = 1$ and such that the restriction of $E$ to $A$ is hyperfinite. It is currently not known whether measure-hyperfiniteness implies hyperfiniteness, and a positive answer would imply that amenability and hyperfiniteness are equivalent, since, by the celebrated Connes-Feldman-Weiss theorem \cite{CFW}, amenability implies measure-hyperfiniteness.

In this paper, we contribute to these lines of inquiry focusing on \emph{treeable} Borel equivalence relations, namely relations that admit acyclic Borel graphings (the study of these relations finds a large space in the literature; see, e.g., \cites{adams:trees, JKL, gaboriau, hjort:treeable, miller:treeable, tree-like}). Typical examples are the orbit equivalence relations originating from free actions of virtually free groups---that is groups that have a free subgroup of finite index---and our first main result confirms hyperfiniteness for a vast class of measure-hyperfinite equivalence relations of this form. In what follows, given a group action $G \curvearrowright X$, we denote by $E_G^X$ the \emph{orbit equivalence relation} associated to the action, that is, $x E_G^X y$ if and only if there is $g \in G$ such that $gx = y$, for $x,y \in X$.

\begin{theoremintro} \label{thm_main:Fk}
Let $F$ be a countable virtually free group, let $X$ be a $\sigma$-compact Polish space, and let $F \curvearrowright X$ be a continuous free action. If $E_{F}^X$ is measure-hyperfinite---in particular, if $E_{F}^X$ is amenable---then it is hyperfinite.
\end{theoremintro}

The reader interested in other known instances of (necessarily amenable) actions of non-amenable groups giving rise to hyperfinite equivalence relations, is referred to the recent string of papers \cites{MarSab, Karp, PrSa, trees, Oyakawa, NV}, where various \emph{ad hoc} classes of actions---mainly boundary actions---are considered.

Theorem \ref{thm_main:Fk} is a corollary of a more general criterion for hyperfiniteness that provides a partial answer to the question posed in \cite[\S 6.4.C]{JKL}, asking whether countable Borel equivalence relations that are both amenable and treeable are hyperfinite. This simpler version of the question whether amenability implies hyperfiniteness is motivated by the well-known fact that hyperifinite countable Borel equivalence relations are treeable. We partially solve this problem under the assumption of a novel strong form of amenability---still implied by hyperfiniteness---which we call \emph{hyper-u-amenability} (Definition \ref{def:hyper-u-amenable}).

\begin{theoremintro} \label{thm_main:hyper_treeable}
Let $E$ be a countable Borel equivalence relation on a standard Borel space. If $E$ is treeable and hyper-u-amenable, then it is hyperfinite.
\end{theoremintro}

Before discussing the concept of hyper-u-amenability in more detail, we present some examples, and some corollaries of Theorem \ref{thm_main:hyper_treeable}.
The first corollary is, as hinted before, Theorem \ref{thm_main:Fk} itself, since all the equivalence relations considered therein are hyper-u-amenable (Corollary \ref{cor: continuous action}).

Our second application relates to a special case of the problem whether amenability implies hyperfiniteness, formulated in a famous question by Weiss in \cite{weiss}, asking if all actions of countable amenable groups induce hyperfinite orbit equivalence relations. Although the general case remains beyond reach, this question has been positively answered for various classes of amenable groups (\cites{weiss, SlamanSteel, JKL, GJ, SS, CJMST-D}).

As we will see in Proposition \ref{prop:group_action}, hyper-u-amenability is satisfied by orbit equivalence relations of Borel actions of countable amenable groups, hence Theorem \ref{thm_main:hyper_treeable} allows to solve affirmatively Weiss' question in the treeable case.

\begin{corollaryintro} \label{cor_intro:group}
Let $G \curvearrowright X$ be a Borel action of a countable amenable group on a standard Borel space. If $E_G^X$ is treeable, then it is hyperfinite.
\end{corollaryintro}

Finally, we show that hyper-u-amenability is automatic for amenable equivalence relations that are \emph{Borel bounded} in the sense of \cite{BJ} (Proposition \ref{prop: Borel bounded}), and obtain the following.

\begin{corollaryintro} \label{cor_intro:bounded}
Let $E$ be a countable Borel equivalence relation on a standard Borel space $X$. If $E$ is treeable, amenable, and Borel bounded, then it is hyperfinite.
\end{corollaryintro}

The concept of hyper-u-amenability is itself based on an auxiliary notion which we call \emph{u-amenability} (Definition \ref{def:uniform_amenability}). U-amenability is modeled on the notion of amenability as defined in \cite[Definitions 2.11 and 2.12]{JKL}, but it applies to equivalence relations $E$ on standard Borel spaces $X$ equipped with an \emph{extended metric} $\rho$---typically originating from a graphing of $E$ (see Section \ref{section:preliminaries})---and depends on the additional structure encoded by $\rho$. More precisely, u-amenability strengthens amenability by requiring \emph{uniform} convergence to zero---hence the prefix `u-'---for the invariance condition on the probability measures assigned to different points in the same orbit (see \S\ref{section: u-amenability} for details).\footnote{We remark that a similar property was independently considered in \cite[Definitions 1.1 and 7.1]{ElekTimar}. The main conclusion of their work is that it is equivalent to a probabilistic version of hyperfiniteness.}

The focus on Borel extended metric spaces is motivated by the breakthrough on Weiss' question obtained in \cite{CJMST-D} with the introduction and study of \emph{Borel asymptotic dimension} (Definition \ref{def: asdim}). This notion, developed in the context of Borel extended metric spaces, not only provides a criterion for hyperfiniteness, but it moreover allows to solve certain instances of the so-called \emph{Union Problem}---which asks whether increasing countable unions of hyperfinite equivalence relations are hyperfinite. Since its introduction, it has been discovered that Borel asymptotic dimension has many more applications than just the Weiss' question (see e.g. \cite{NV}), and it plays a major role in this paper as well.
 
Theorem \ref{thm_main:hyper_treeable} is in fact a consequence of the more technical Theorem \ref{thm:ua_asdim}, which applies to countable Borel equivalence relations $E$ on a standard Borel space $X$  that admit an acyclic Borel graphing $\cG \coloneqq (X, R)$ that has a finite uniform bound on the degree of all vertices. More precisely, Theorem \ref{thm:ua_asdim} asserts that if such an $E$ is u-amenable with respect to the shortest path metric $\rho_\cG$ on $X$, then the Borel asymptotic dimension of $(X, \rho_\cG)$ is finite hence, in particular, $E$ is hyperfinite by \cite[Theorem 1.7]{CJMST-D}.

Broadly speaking, to prove Theorem \ref{thm:ua_asdim} we exploit u-amenability to define a \emph{partial} Borel orientation on the acyclic graphing $\cG$, which can then be used to show $\asdim(X, \rho_\cG) < \infty$. The idea of using amenability to define a Borel orientation on a graph has precursors in the measure-theoretic setting. More specifically,  if the Borel space in question is equipped with a probability measure, then any amenable equivalence relation is hyperfinite on a {set of full measure} by \cite{CFW} and, if the relation is additionally equipped with the structure of an acyclic graph, amenability (or hyperfiniteness) is also equivalent to the so-called \emph{end selection}. This means that for each connected component (which has the structure of a connected tree) one can select either a point in it or at most two points from its Gromov boundary, and this selection is measurable (see \cite{adams:trees} and \cite[\S 3.6]{JKL}). Note however that a direct application of these methods is not possible in the setting considered in this paper: the equivalence between end selection and hyperfiniteness fails in the \emph{purely Borel} setting, as there are hyperfinite acyclic Borel graphs for which it is impossible to select at most two ends from each component in a Borel way (\cite[Remark 5.11]{miller:ends}). The Borel orientations we shall build are, in fact, only partial. Additionally, the techniques used in the aforementioned results rely on taking a Banach limit, which cannot be done in a Borel way.

One disadvantage of u-amenability is that it depends on the metric in question---in particular, an equivalence relation might have, a priori, one graphing which is u-amenable and another which is not. It turns out, however, that if an equivalence relation $E$ has a graphing that is an \emph{increasing union of graphs generating u-amenable equivalence relations}, then the same holds for any other graphing of $E$ (Proposition \ref{prop: increasing union graphings}). This can thus be regarded as an intrinsic property of the equivalence relation itself, and we refer to it as \emph{hyper-u-amenability} (Definition \ref{def:hyper-u-amenable}).

This last notion brings us back to the most general version of our criterion for hyperfiniteness, Theorem \ref{thm_main:hyper_treeable}, which follows from Theorem \ref{thm:ua_asdim} exploiting the aforementioned key fact that the finite Borel asymptotic dimension permits to solve certain instances of the Union Problem (see Proposition \ref{prop:union}).

\subsection*{Summary} The paper is organized as follows. In Section \ref{section:preliminaries} we cover the preliminaries, while in Section \ref{section: u-amenability} we introduce u-amenability and hyper-u-amenability, and provide some examples of equivalence relations with these properties.
In Section \ref{section:partial_orientations} we prove some intermediate lemmas and the main criterion needed to prove Theorem \ref{thm:ua_asdim}, whose proof is presented in Section \ref{section:asdim}, along with the proofs of Theorem \ref{thm_main:hyper_treeable} (Corollary \ref{cor: main result}), Theorem \ref{thm_main:Fk}, and Corollaries \ref{cor_intro:group} and \ref{cor_intro:bounded} (Corollary \ref{cor:applications}).

\subsection*{Acknowledgements} We thank Marcin Sabok and Anush Tserunyan for their helpful remarks on an earlier version of this work, which revealed that our methods leading to Theorem \ref{thm_main:hyper_treeable} also apply to actions of free groups with infinitely many generators and more generally of virtually free groups. We also thank Zoltán Vidnyánszky for helpful discussions and suggestions. We are finally grateful to the referee for their useful and pertinent remarks which sensibly improved the readability of the paper.

\section{Preliminaries} \label{section:preliminaries}
\subsection{Borel equivalence relations and Borel graphs} \label{ss:Borel} Let $X$ be a standard Borel space. An equivalence relation $E$ on $X$ is \emph{Borel} if $E \subseteq X^2$ is Borel.
The \emph{$E$-class} of a point $x \in X$ is denoted $[x]_E$.
A Borel equivalence relation is \emph{finite} (respectively \emph{countable}) if all its equivalence classes are finite (respectively countable), and it is \emph{hyperfinite} if it can be written as union of an increasing sequence of finite Borel equivalence relations. Finally, a  countable Borel equivalence relation $E$ over $X$ is \emph{measure-hyperfinite} if, for every probability measure $\mu$ over $X$, there exists some Borel subset $A \subseteq X$, with $\mu(A) = 1$, such that the restriction of $E$ to $A$ is hyperfinite.

Let $X$ be a standard Borel space and let $\rho$ be a \emph{Borel extended metric} on $X$, that is a Borel metric $\rho \colon X \times X \to [0, \infty]$ that can also take value $\infty$. In this case, we call $(X, \rho)$ a \emph{Borel extended metric space}. For $A \subseteq X$ and $x \in X$, we define
\begin{equation*}
\rho(x,A)\coloneqq \inf_{y \in A} \rho(x,y).
\end{equation*}
For $r > 0$ and $x \in X$, we let $B_\rho(x, r)$ denote the closed ball of radius $r$ centered in $x$ with respect to $\rho$. A Borel extended metric space is \emph{proper} if $| B_\rho(x,r) |< \infty$ for all $x \in X$ and $r > 0$. %Similarly, for a set $A \subseteq X$, we define
%\begin{equation*}
%B_\rho(A,r) \coloneqq \{ x\in X : \rho(x, A) \le r \}.
%\end{equation*}
Denote by $E_\rho$ the Borel equivalence relation of finite distance $\rho$ between points in $X$, that is
\begin{equation*}
E_\rho \coloneqq \{ (x,y) \in X^2 : \rho(x,y) < \infty \}.
\end{equation*}
An equivalence relation $E$ on a Borel extended metric space $(X, \rho)$ is \emph{$\rho$-uniformly bounded} if there exists $r > 0$ such that
every $E$-class has $\rho$-diameter smaller than $r$.

Given a subset $R \subseteq X^2$, we let $R^{-1} \coloneqq \{ (x,y) \in X^2 : (y,x) \in R \}$ and we write $xRy$ as an abbreviation for $(x,y) \in R$, and $x \neg R y$ for $(x,y) \notin R$. A \emph{Borel graph} on $X$ is a graph $\cG \coloneqq (X, R)$ where $R$ is a Borel subset of $X^2$ that is \emph{symmetric}, that is $R = R^{-1}$, and \emph{irreflexive}, meaning that $x \neg R x$ for all $x \in R$. Two vertices $x,y \in X$ are \emph{$R$-adjacent}
if $xRy$. Given a vertex $x \in X$ and an edge $(y,z) \in R$, we say that they are \emph{$R$-adjacent} if either $x=y$ or $x =z$. In both cases, we might drop `$R$-' and simply write \emph{adjacent} if no confusion arises.

If $\cG$ is a Borel graph on $X$, let $\rho_\cG$ denote the Borel extended metric on $X$ induced by the path distance, and abbreviate $E_{\rho_\cG}$ as $E_\cG$.
In this case, the equivalence class $[x]_{E_\cG}$ of $x \in X$, which we abbreviate as $[x]_\cG$, is the connected component of $x$ in $\cG$. Given $q_0, q_1 \in R$, say $q_0 = (x_0, x_1)$ and $q_1 = (y_0, y_1)$ we define
\begin{equation} \label{eq:edge_distance}
\rho_\cG(q_0, q_1) \coloneqq \min_{i,j \le 1} \rho_\cG(x_i, y_j).
\end{equation}
%We abbreviate $B_{\rho_\cG}(x,r)$ as $B_{\cG}(x, r)$, and $B_{\rho_\cG}(A,r)$ as $B_{\cG}(A, r)$.

For a vertex $x \in X$, the \emph{degree} of $x$ is the value
\begin{equation*} 
\deg_\cG(x) \coloneqq |\{ y \in X : xRy\}|.
\end{equation*}
The graph $\cG$ is locally finite if $\deg_\cG(x) < \infty$ for all $x \in X$.
Vertices of degree 1 are called \emph{leaves} and the set of all leaves in $\cG$ is denoted $\leaves(\cG)$. The \emph{degree} of $\cG$ is
\begin{equation*}
\deg(\cG) \coloneqq \sup_{x \in X} \deg_\cG(x).
\end{equation*}
%We say that a Borel graph $\cG$ has \emph{finite degree} if $\deg(\cG) < \infty$.

A \emph{(Borel) graphing} of a Borel equivalence relation $E \subseteq X^2$ is a Borel graph $\cG$ on $X$ such that $E_\cG = E$. We say that $E$ is \emph{treeable} if there exists an acyclic graphing of $E$. The \emph{complete graphing} of $E$ is the Borel graph $(X, E \setminus \Delta(X))$, where $\Delta(X) = \{(x, x) \in X^2 : x \in X\}$ is the diagonal of $X^2$.

Finally, for a subset $Y \subseteq X$ and $R \subseteq X^2$, we write $R|_Y$ for the restriction $R \cap Y^2$, and for a graph $\cG \coloneqq (X, R)$, we write $\cG|_Y$ for the graph $(Y, R|_Y)$.

\subsection{Borel orientations and out-degree}

A \emph{Borel orientation} of $\cG= (X, R)$ is a Borel subset $\vv{R} \subseteq R$ such that $\vv{R} \cap \vv{R}^{-1} = \emptyset$ and $\vv{R} \cup \vv{R}^{-1} = R$. In other words, for every $xRy$, exactly one among
$x\vv{R}y$ and $y\vv{R}x$ holds. The \emph{out-degree} of a vertex $x \in X$ is the value
\begin{equation*} 
\degout(x) \coloneqq \abs{\left\{y \in X : x\vv{R} y\right\}}.
\end{equation*}
We also define the \emph{out-degree} of the pair $(X, \vv{R})$ as
\begin{equation*}
\Degout(X,\vv{R}) \coloneqq \sup_{x \in X}\degout(x). 
\end{equation*}

\begin{remark} \label{remark:degree_maps}
Let $\cG \coloneqq (X, R)$ be a Borel graph. Note that, in case $\deg(\cG) \in \bbN \cup \{\aleph_0\}$, a standard application of the Lusin--Novikov uniformization theorem (\cite[Theorem 18.10]{Kechris:CDST}) shows that the maps
\begin{equation*}
\deg_\cG \colon X \to \bbN \cup \{\aleph_0\}, \quad \degout \colon X \to \bbN \cup \{\aleph_0\}
\end{equation*}
are Borel. In particular, $\leaves(\cG) = (\deg_\cG)^{-1}(1)$ is a Borel subset of $X$.
Note moreover that, given a constant $C > 0$, the set
\begin{equation*}
Y \coloneqq \{ x \in X : |[x]_\cG | \le C \}
\end{equation*}
is Borel.
Indeed, if $\cG_0$ is the complete graphing of $E_\cG$, then $Y = (\deg_{\cG_0})^{-1}([0, C)])$.
\end{remark}

\subsection{Borel asymptotic dimension}
We recall the definition of Borel asymptotic dimension from \cite{CJMST-D}.

\begin{definition}[{\cite[Definition 3.2]{CJMST-D}}]\label{def: asdim}
  Let $(X,\rho)$ be a Borel extended metric space such that $E_\rho$
    is countable. The \emph{Borel
    asymptotic dimension} of $(X, \rho$), denoted $\asdim(X,\rho)$, is the smallest $d \in \bbN$
    such that, for every $r > 0$, there exists a 
    $\rho$-uniformly bounded Borel equivalence relation $E$ on $X$ with the property that, for
    every $x \in X$, the ball $B_\rho(x,r)$ intersects at most $d+1$ classes of $E$,
    and it is $\infty$ if no such $d$ exists.
\end{definition}

Finite Borel asymptotic dimension is one of the main tools for accessing hyperfiniteness in this paper. Finite Borel asymptotic dimension provides, moreover, a tame framework where certain instances of the Union Problem can be solved affirmatively. This is synthesized in a precise manner in \cite[Theorem 7.3]{CJMST-D}. The proposition below isolates a specific instance of this result, in the context of Borel graphs.

\begin{proposition}\label{prop:union}
Let $\cG_n \coloneqq (X, R_n)$ be locally finite Borel graphs and $R_n \subseteq R_{n+1}$, for all $n \in \bbN$. Suppose that $\asdim (X, \rho_{\cG_n}) < \infty$ for all $n \in \bbN$. Then $E \coloneqq \bigcup_{n \in \bbN} E_{\cG_n}$ is hyperfinite.
\end{proposition}
\begin{proof}
Since $R_n \subseteq R_{n+1}$, we have
\[
\rho_{\cG_{n+1}} \le \rho_{\cG_n}.
\]
The conclusion then directly follows from \cite[Theorem 7.3]{CJMST-D}.
\end{proof}

\subsection{Schreier graphs and orbit equivalence relations}
Let $G$ be a group, let $X$ be a standard Borel space, and let $G \curvearrowright X$ be an action. The \emph{orbit equivalence relation} $E_G^X$ on $X$ is defined as
\begin{equation*}
x E_G^X y \text{ if and only if there is }g \in G \text{ such that }gx = y.
\end{equation*}
The relation $E_G^X$ is countable if $G$ is countable, and it is Borel if $G \curvearrowright X$ is a {Borel action}. We recall that if $G$ is a \emph{virtually free group}---meaning that $G$ has a free subgroup of finite index---and the action $G \curvearrowright X$ is free, then $E_G^X$ is treeable (see e.g. \cite[Theorem 55 on p. 240]{cohen:cgt}).

If $S$ is a subset of $G$, the induced \emph{Schreier graph} $ \sch(X, S)$ on $X$ is defined by putting an edge between any distinct $x,y \in X$ such that there is $s \in S \cup S^{-1} \setminus \{ e \}$ with $sx = y$. If $S$ is a generating set for $G$ then $E_{\sch(X,S)} = E_G^X$.

\section{U-amenability and hyper-u-amenability}\label{section: u-amenability}
In this section we introduce \emph{u-amenability} and \emph{hyper-u-amenability} and provide some examples of equivalence relations with these properties.

\begin{definition} \label{def:uniform_amenability}
Let $(X,\rho)$ be a Borel extended metric space and let $E$ be a countable Borel equivalence relation on $X$.
We say that $E$ is \emph{u-amenable with respect to $\rho$} if  $E \subseteq E_\rho$ and if there are Borel maps
\begin{equation*}
\lambda_n \colon E \to \mathbb{R}_{\ge 0}, \quad n \in \bbN,
\end{equation*}
such that, with the abbreviation $\lambda_{n,x}(\cdot) \coloneqq \lambda_n(x, \cdot)$,
\begin{enumerate}
\item $\lambda_{n,x} \in \ell^1([x]_{E})$ and $\| \lambda_{n,x} \|_1 = 1$ for every $n \in \bbN$ and $x \in X$,
\item \label{item2:ua} $\sup_{\{ (x,y) \in E  : \rho(x,y) < r \}} \| \lambda_{n,x} - \lambda_{n,y} \|_1 \to 0$ as $n \to \infty$, for every $r > 0$.
\end{enumerate}

A Borel graph $\cG \coloneqq (X,R)$ is \emph{u-amenable} if $E_\cG$ is u-amenable with respect to $\rho_\cG$. Note that, in this case, \eqref{item2:ua} is equivalent to
\begin{equation*}
\sup_{ (x,y) \in R } \| \lambda_{n,x} - \lambda_{n,y} \|_1 \to 0, \quad n \to \infty.
\end{equation*}
\end{definition}

The prefix `u-' stands for \emph{uniform}, and emphasizes the difference between our definition and the more well-known \emph{1-amenability}---or simply \emph{amenability}---for a countable Borel equivalence relation $E$ introduced in \cite[Definition 2.12]{JKL}, which requires the existence of Borel functions $\lambda_n \colon E \to \mathbb{R}_{\ge 0}$ as in Definition \ref{def:uniform_amenability}, but where condition \eqref{item2:ua} above is replaced by a weaker and, so to speak, \emph{pointwise} condition. We recall it below.
\begin{definition}[{\cite[Definitions 2.11 and 2.12]{JKL}}] \label{def:amenability}
Let $X$ be a standard Borel space and let $E$ be a countable Borel equivalence relation on $X$.
We say that $E$ is \emph{1-amenable}, or simply \emph{amenable}, if there are Borel maps
\begin{equation*}
\lambda_n \colon E \to \mathbb{R}_{\ge 0}, \quad n \in \bbN,
\end{equation*}
such that, with the convention $\lambda_{n,x}(\cdot) \coloneqq \lambda_n(x, \cdot)$,
\begin{enumerate}
\item $\lambda_{n,x} \in \ell^1([x]_{E})$ and $\| \lambda_{n,x} \|_1 = 1$ for every $n \in \bbN$ and $x \in X$,
\item $\| \lambda_{n,x} - \lambda_{n,y} \|_1 \to 0$, as $n \to \infty$, whenever $xEy$.
\end{enumerate}
\end{definition}

The next lemma, which we isolate for later use, implies that u-amenability is preserved when taking sub-equivalence relations or subgraphs.
\begin{lemma}\label{lemma:subgraph_ua}
Let $X$ be a standard Borel space and let $E$ be a countable Borel equivalence relation on $X$. Suppose that $X_0 \subseteq X$ is a Borel subset, that $\rho_0$ is an extended Borel metric on $X_0$, and let $E_0$  be a countable Borel equivalence relation on $X$ such that $E_0 \subseteq E \cap E_{\rho_0}$.
Suppose finally that there are Borel maps
\begin{equation*}
\lambda_n \colon E \to \bbR_{\ge 0}, \quad n \in \bbN,
\end{equation*}
such that $\lambda_{n,x} \in \ell^1([x]_{E})$, $\| \lambda_{n,x} \|_1 = 1$ and, for every $r > 0$,
\begin{equation}\label{eq:ua_Y}
\sup_{\{ (x,y) \in E_0  : \rho_0(x,y) < r \}} \| \lambda_{n,x} - \lambda_{n,y} \|_1 \to 0, \quad n \to \infty.
\end{equation}
Then $E_0$ is u-amenable with respect to $\rho_0$. In particular, if $\cG \coloneqq (X, R)$ is a u-amenable Borel graph and $\cG_0 \coloneqq (X_0, R_0)$ is a Borel graph with $R_0 \subseteq R$, then $\cG_0$ is also u-amenable.
\end{lemma}

\begin{proof}
Use the Lusin--Novikov uniformization theorem (\cite[Theorem 18.10]{Kechris:CDST}) to find Borel functions $F_k \colon X \to X$, for $k\in \bbN$, such that $E = \{(x,y) \in X^2 : \exists k \, (y = F_k(x)) \}$. Consider the following two Borel functions
\begin{equation*}
\begin{aligned}[c]
N \colon E & \to \bbN  \\
(x,y) &\mapsto \min \{k : x E_0 F_k(y) \},
\end{aligned}
\qquad
\begin{aligned}[c]
P \colon E &\to X \\
(x,y) &\mapsto F_{N(x,y)}(y).
\end{aligned}
\end{equation*}
%Both functions are Borel: for every $k \in \bbN$, the function $H_k\colon E_\cG \to \bbN \cup \{ \infty \}$ that is constantly $k$ on the preimage of $E_\cG$ via the Borel map $(x,y) \mapsto (x, F_k(y))$ and $\infty$ elsewhere is Borel, and since $N= \inf_{k \in \bbN} H_k$, we have that $N$ is Borel as well.
For every $n \in \bbN$, define finally the Borel functions
\begin{align*}
\lambda'_n \colon E_0 &\to \bbR_{\ge 0} \\
(x,y) & \mapsto   \sum_{\substack{z \in [x]_{E}, \\ P(x, z) = y }}\lambda_{n,x}(z). 
\end{align*}

The sequence $(\lambda'_n)_{n \in \bbN}$ witnesses that $E_0$ is u-amenable with respect to $\rho_0$.
Indeed, given $n\in \bbN$ and $x \in Y$, we have
\begin{equation*}
\| \lambda'_{n,x} \|_1 = \sum_{y \in [x]_{E_0}} \lambda'_{n,x}(y) = \sum_{y \in [x]_{E_0}}  \sum_{\substack{z \in [x]_{E}, \\ P(x, z) = y }}\lambda_n(x, z) = \sum_{y \in [x]_{E}} \lambda_{n,x}(y) = \| \lambda_{n,x} \|_1 = 1.
\end{equation*}

Finally, to verify that the sequence $(\lambda'_n)_{n \in \bbN}$ satisfies item \eqref{item2:ua} of Definition \ref{def:uniform_amenability}, note that if $xE_0y$ and $w \in [x]_{E}$ then $N(x, w) = N(y,w)$ by definition of $N$, hence $P(x,w) = P(y, w)$. This, in particular, implies that
\begin{equation} \label{eq:lambda_n}
 \sum_{\substack{w \in [x]_{{E}}, \\ P(x,w) = z}} \lambda_{n,x}(w) - \sum_{\substack{\bar w \in [y]_{{E}}, \\ P(y,\bar w) = z}} \lambda_{n,y} (\bar w) =
  \sum_{\substack{w \in [x]_{{E}}, \\ P(x,w) = z}} \lambda_{n,x}(w) - \lambda_{n,y}(w),
  \end{equation}
which in turn gives
\begin{align*}
\| \lambda'_{n,x} - \lambda'_{n,y} \|_1 & \stackrel{\mathclap{\eqref{eq:lambda_n}}}{\le} \sum_{z \in [x]_{E_0}}  \sum_{\substack{w \in [x]_{E}, \\ P(x,w) = z}} | \lambda_{n,x}(w) - \lambda_{n,y}(w)| \\
&= \sum_{z \in [x]_{E_0}} | \lambda_{n,x}(z) - \lambda_{n,y}(z)| \\
&= \| \lambda_{n,x} - \lambda_{n,y} \|_1
\end{align*}
Item \eqref{item2:ua} of Definition \ref{def:uniform_amenability} for $(\lambda'_n)_{n \in \bbN}$ hence follows by \eqref{eq:ua_Y}.
\end{proof}
%Note that an action of a countable group $G \curvearrowright X$ on a $\sigma$-compact Polish space is topologically amenable if and only if it is amenable in the sense of \cite[Definition 2.12]{DJK} (see \cite[Theorem A.3.1]{FKSV}). In particular, Proposition \ref{prop:topologically_amenable} also applies to continuous actions that are amenable in the following sense.
\begin{definition} \label{def:hyper-u-amenable}
Let $X$ be a standard Borel space. A countable Borel equivalence relation $E$ on $X$ is \emph{hyper-u-amenable} if the complete graphing of $E$ is an increasing union of Borel u-amenable graphs on $X$, that is, if that there are u-amenable Borel graphs $\cG_n \coloneqq (X, R_n)$, for $n \in \bbN$, such that $R_n \subseteq R_{n+1}$ and $E = \bigcup_{n \in \bbN} R_n$.
\end{definition}

We show in the next proposition that, if $E$ admits a graphing which is an increasing union of u-amenable graphs, then all graphings of $E$, including the complete one, have this property. We moreover prove that the u-amenable graphs can always be assumed to have finite degree.

\begin{proposition}\label{prop: increasing union graphings}
Let $X$ be a standard Borel space and let $E \subseteq X^2$ be a countable Borel equivalence relation. Suppose that $E$ has a Borel graphing $\cH$ which is increasing union of u-amenable Borel graphs. If $\cG$ is any other graphing of $E$, then $\cG$ is an increasing union of u-amenable Borel graphs of finite degree. In particular, $E$ is hyper-u-amenable.
\end{proposition}

\begin{proof}
Let $\cH \coloneqq (X, S)$ be a graphing witnessing hyper-u-amenability of $E$. There exist therefore, for $n \in \bbN$, u-amenable Borel graphs  $\cH_n \coloneqq (X, S_n)$ such that
\[
S = \bigcup_{n=1}^\infty S_n, \quad S_n \subseteq S_{n+1}, \quad n \in \bbN.
\]

Without loss of generality, we may assume that each $\cH_n$ has finite degree. To see this, note that by \cite{FM} there is a countable group $G$, enumerated as $\{ g_n \} _{ n\in \bbN}$, and a Borel action $G \curvearrowright X$ such that $E = E_G^X$. Denote by $T_n$ the edge relation of the Schreier graph $\sch(X, \{ g_0, \dots, g_{n} \})$. Given $n \in \bbN$, the graph $\cH'_n \coloneqq (X, S_n \cap T_n)$ has finite degree since $\sch(X, \{ g_0, \dots, g_{n} \})$ does, it is u-amenable by Lemma \ref{lemma:subgraph_ua} since $\cH_n$ is, and we moreover have both $S_n \cap T_n \subseteq S_{n+1} \cap T_{n+1}$ and $S = \bigcup_{n \in \bbN} (S_n \cap T_n)$.

Suppose next that $\cG \coloneqq (X, R)$ is another graphing of $E$, define 
\[
R_n  \coloneqq \{(x, y) \in R : \rho_{\cH_n}(x, y) \le n\}, \quad n \in \bbN,
\]
and set $\cG_n \coloneqq (X, R_n)$. We show that $\cG$ is an increasing union of the $\cG_n$'s, and that these are u-amenable and with finite degree.

The inclusion $R_n \subseteq R_{n+1}$ follows since
\begin{equation}\label{eq: distance G_n}
\rho_{\cH_{n+1}} \le \rho_{\cH_n}.
\end{equation}
Since $E = E_\cG = E_\cH$, for every edge $xRy$ there is a sufficiently large $n$ such that $xE_{\cH_n}y$, that is, $\rho_{\cH_n}(x, y) < \infty$. By \eqref{eq: distance G_n}, it then follows that $x R_m y$ with $m = \max\{ n, \rho_{\cH_n}(x, y) \}$, and thus $R = \bigcup_{n \in \bbN} R_n$.
Moreover, the inequality
\begin{equation}\label{eq: distance R_n}
\rho_{\cG_n} \ge \frac{1}{n}\rho_{\cH_n}.
\end{equation}
implies $\deg(\cG_n) \le \deg(\cH_n)^n < \infty$. Finally, \eqref{eq: distance R_n} and Lemma \ref{lemma:subgraph_ua}, applied to  $E_{\cH_n}$ and $E_{\cG_n}$,
imply that the graphs $\cG_n$ are u-amenable.
\end{proof}

\subsection{Examples} 
\subsubsection{Topologically amenable actions}
The first class of u-amenable equivalence relations that we consider comes from continuous, \emph{topologically amenable} actions of finitely generated groups on compact spaces, of which we briefly recall the definition.
\begin{definition}
Let $G$ be a countable group and $X$ a Polish space. A continuous action $G \curvearrowright X$ is \emph{topologically amenable} if for every compact subset $K \subseteq X$ there are continuous maps
\begin{equation*}
p_n \colon X \times G \to \bbR_{\ge 0}, \quad n \in \bbN,
\end{equation*}
such that, with the convention $p_{n,x}(\cdot) \coloneqq p_n(x, \cdot)$,
\begin{enumerate}
\item $p_{n,x} \in \ell^1(G)$ and $\| p_{n,x} \|_1 = 1$ for every $n \in \bbN$ and $x\in X$,
\item $\sup_{x \in K} \| g\cdot p_{n,x} - p_{n,gx} \|_1 \to 0$ as $n \to \infty$, for every $g \in G$, and where $g \cdot p_{n,x}(h) = p_{n,x}(g^{-1}h)$.
\end{enumerate}
\end{definition}

\begin{proposition} \label{prop:topologically_amenable}
Let $G$ be a countable group and $X$ a Polish space. Suppose that $G \curvearrowright X$ is a continuous, topologically amenable action, let $K \subseteq X$ be a compact set, and let $S \subseteq G$ be finite. Then the graph $\sch(X, S)|_K$ is u-amenable.
\end{proposition}
\begin{proof}
Let $(p_n)_{n \in \bbN}$ is a sequence of maps witnessing topological amenability of $G \curvearrowright X$ for $K \subseteq X$, and define the Borel maps
\begin{align*}
\lambda_n \colon E_G^X &\to \bbR_{\ge 0} \\
(x,y) &\mapsto \sum_{\substack{g \in G \\ g^{-1}x = y}} p_{n,x}(g).
\end{align*}
Then the sequence $\lambda_n$ satisfies the assumption of Lemma \ref{lemma:subgraph_ua} with $E = E_G^X$, $E_0 = E_{\sch(X, S)|_K}$ and $\rho_0 = \rho_{\sch(X, S)|_K}$, hence the graph $\sch(X, S)|_K$ is u-amenable.
\end{proof}

We report a result from \cite[Appendix A]{FKSV} which, when combined with the Connes--Feldman--Weiss theorem, relates topological amenability to measure-hyperfiniteness of $E_G^X$ in the case of continuous actions on Polish spaces. 

\begin{theorem}[{\cite{CFW},\cite[Appendix A]{FKSV}}] \label{theorem:appendix}
Let $G \curvearrowright X$ be a continuous action of a countable group on a Polish space. The following are equivalent:
\begin{enumerate}
\item $G \curvearrowright X$ is topologically amenable,
\item $E_G^X$ is measure-hyperfinite and every stabilizer is amenable.
\end{enumerate}
\end{theorem}

We use the equivalence above to prove the following corollary of Proposition \ref{prop:topologically_amenable}.

\begin{corollary}\label{cor: continuous action}
Let $G \curvearrowright X$ be a continuous action of a countable group on a $\sigma$-compact Polish space. If the stabilizer of every point is amenable and $E_G^X$ is measure-hyperfinite, then $E_G^X$ is hyper-u-amenable.  In particular,  $E_G^X$ is hyper-u-amenable whenever $E_G^X$ is measure-hyperfinite (e.g. if $E_G^X$ is amenable) and the action is free.
\end{corollary}

\begin{proof}
By Theorem \ref{theorem:appendix}, the action $G \curvearrowright X$ is topologically amenable. Since $X$ is $\sigma$-compact, there is an increasing sequence of compact sets $K_n$ such that $X = \bigcup_{n \in \bbN} K_n$. Let $\{ g_n \}_{n \in \bbN}$ be an enumeration of $G$ and define
\[
\cG_n \coloneqq \sch(X, \{g_0, \ldots, g_n\})|_{K_n}, \quad n \in \bbN.
\]
Each $\cG_n$ is u-amenable by Proposition \ref{prop:topologically_amenable}, and since the sequence $(\cG_n)_{n \in \bbN}$ is increasing, and its union is the complete graphing of $E_G^X$, it follows that $E_G^X$ is hyper-u-amenable.
\end{proof}

\subsubsection{Actions of amenable groups}
\begin{proposition} \label{prop:group_action}
Let $G$ be a countable amenable group, let $X$ be a standard Borel space, and let $G \curvearrowright X$ be a Borel action. Then $E_G^X$ is hyper-u-amenable. 
\end{proposition}

\begin{proof}
Let $\{ g_n \}_{n \in \bbN}$ be an enumeration of $G$ and let $\cG \coloneqq (X, E_G^X \setminus \Delta(X))$ be the complete graphing of $E$. For every $m\in \bbN$ consider the subgroup $G_m \coloneqq \langle g_0, \dots, g_m \rangle$ and the corresponding Schreier graph $\cG_m \coloneqq \sch(X, \{g_0, \ldots, g_m\})$. Denote the edge relation on $\cG_m$ as $R_m$. Note that the graphs $\cG_m$ have finite degree, that $R_m \subseteq R_{m+1}$ for all $m \in \bbN$, and that $E_G^X \setminus \Delta(X) = \bigcup_{m \in \bbN} R_m$. To see that $E_G^X$ is hyper-u-amenable, it remains to show that all $\cG_m$'s are u-amenable.

To do this, fix $m \in \bbN$, and let $(F_{n})_{n \in \bbN}$ be a F{\o}lner sequence for the subgroup $G_m$. U-amenability of $\cG_m$ is then witnessed by the family $(\lambda_n)_{n \in \bbN}$ of Borel functions defined as \begin{align*}
\lambda_n \colon E_{\cG_m} &\to \bbR_{\ge 0} \\
(x,y) &\mapsto \frac{1}{|F_n|} \abs{ \{ g \in F_n : gx = y \}}.
\end{align*}
\end{proof}

We remark that the proof above actually shows that if $G$ is a finitely generated amenable group and $S$ is some finite generating set, then the Schreier graph $\sch(X, S)$ is u-amenable.

\subsubsection{Borel bounded equivalence relations} The last class of examples of hyper-u-amenable equivalence relations that we provide is that of amenable Borel bounded equivalence relations, of which we recall the definition below. For $f , g \in \bbN^\bbN$, we write $f=^\ast g$ (and $f \le ^\ast g$) if $f(n) = g(n)$ (respectively $f(n) \le g(n)$) for all but finitely many $n \in \bbN$. 
\begin{definition}[{\cite{BJ}}]
Let $X$ be a standard Borel space. A countable Borel equivalence relation $E \subseteq X^2$ is \emph{Borel bounded}, if for every Borel function $\phi \colon X \to \bbN^\bbN$, there is a function $\psi \colon X \to \bbN^\bbN$ such that $\phi(x) \le^* \psi(x)$ for all $x \in X$, and that $\psi(x) =^*\psi(y)$ whenever $xEy$.
\end{definition}

Hyperfinite equivalence relations are Borel bounded (\cite{BJ}) and, remarkably, it is not known whether \emph{every} countable Borel equivalence relation is Borel bounded, although Thomas proved in \cite[Theorem 5.2]{thomas} that equivalence relations that are not Borel bounded exist, if Martin's conjecture on degree invariant Borel maps holds.

\begin{proposition}\label{prop: Borel bounded}
Let $X$ be a standard Borel space and let $E \subseteq X^2$ be an amenable, Borel bounded, countable Borel equivalence relation. Then $E$ is hyper-u-amenable. 
\end{proposition}

\begin{proof}
Let $\lambda_n \colon E \to \bbR_{\ge 0}$ be a sequence of Borel functions witnessing amenability of $E$, so in particular $\lambda_{n,x} \in \ell^1([x]_{E})$ and $\| \lambda_{n,x} \|_1 = 1$ for every $n \in \bbN$ and $x \in X$, and
\[
\| \lambda_{n,x} - \lambda_{n,y} \|_1 \to 0, \quad n \to \infty,  \text{ whenever }xEy.
\]

Let $G$ be a countable group, enumerated as $\{ g_n \}_{n\in \bbN}$, along with a Borel action $G \curvearrowright X$ such that $E = E_G^X$ (\cite{FM}). Define the Borel function $\phi \colon X \to \bbN^ \bbN$ as
\[
\phi(x)(n) \coloneqq \min\left\{m \in \bbN : \forall j \ge m, \, \forall  i \le n \text{ if }  (x, g_i x) \in R  \text{ then }  \|\lambda_{j, x} - \lambda_{j, g_i x}\|_1 < 1/n \right\}.
\]
By Borel boundedness there exists a Borel map $\psi : X \to \bbN^\bbN$ such that $\psi(x) =^* \psi(y)$ if $x E y$, and that $\phi(x) \le^* \psi(x)$ for all $x \in X$. 

For $k \in \mathbb{N}$, define a graph $\cG_k \coloneqq (X, R_k)$ by setting $x R_k y$ if
\begin{enumerate}[label=(\alph*)]
\item $y = g_ix$ and $x = g_jy$, for some $i, j \le k$,
\item $\psi(x)(n) = \psi(y)(n)$, for all $n \ge k$,
\item $\max \{\phi(x)(n), \phi(y)(n)\} \le \psi(x)(n)$, for all $n \ge k$.
\end{enumerate}

It is clear that $R_k \subseteq R_{k+1}$ for all $k \in \bbN$ and that $E = \bigcup_{k \in \bbN} R_k$. It remains to show that each subgraph $\cG_k$ is u-amenable. To do so, given $n \in \mathbb{N}$ and $x \in X$, define  $\lambda^\prime_{n, x} \coloneqq \lambda_{\psi(x)(n), x}$. 
We claim that this sequence witnesses u-amenability of $\cG_k$, for any $k\in \bbN$. Indeed, note that if $x R_k y$ and $n \ge k$, then $\psi(x)(n) = \psi(y)(n)$ and $\phi(x)(n) \le \psi(x)(n)$, hence
\[
\|\lambda^\prime_{n, x} - \lambda^\prime_{n, y}\|_1 = \|\lambda_{\psi(x)(n), x} - \lambda_{\psi(x)(n), y}\|_1 < \frac1n.
\]
By the triangle inequality, we thus get that 
\[
 \|\lambda^\prime_{n, x} - \lambda^\prime_{n, y}\|_1 \le \frac{r}{n}, \quad n \ge k, \, x, y \in X \text{ s.t. } \rho_{\cG_k}(x,y) \le r,
 \]
 and therefore $\sup_{\{ (x,y) \in E_{\cG_k}  : \rho_{\cG_k}(x,y) < r \}} \| \lambda'_{n,x} - \lambda'_{n,y} \|_1 \to 0$ as $n \to \infty$, for every $r > 0$.
\end{proof}

\section{Graphs with partial orientations} \label{section:partial_orientations}
Given a Borel graph $\cG \coloneqq (X, R)$ with finite degree, in this section we provide technical conditions sufficient to deduce $\asdim(X, \rho_\cG) < \infty$.

This is done gradually: we first consider the case where there is a Borel orientation $\vv{R} \subseteq R$ with  $\Degout(X, \vv{R}) \le 1$, in which case a direct application of some results in \cite{CJMST-D} implies that $\asdim (X,\rho_\cG) \le 1$ (Proposition \ref{prop:oriented}). We then relax our assumptions and show that $\asdim(X, \rho_\cG) \le 3$ if $R$ admits partial Borel orientations such that the non-oriented edges can be chosen to be arbitrarily far away from each other (Lemma \ref{lemma:delete_edges}).

Finally, in Lemma \ref{lemma:quasi-oriented} we relax even further our assumptions and prove the main technical result of the paper: $\asdim(X, \rho_\cG) \le 3$ holds if $R$ can be partitioned in two sets, $R_0 \sqcup R_1$, such that $R_0$ admits a Borel orientation with out-degree at most 1 and $\cG_1 \coloneqq (X, R_1)$ is an acyclic Borel graph with $\deg(\cG_1) \le 2$ verifying a series of other conditions. The resulting criterion will be essential to deduce finite Borel asymptotic dimension for treeable u-amenable equivalence relations in Theorem \ref{thm:ua_asdim}.

The intended setup is that of acyclic Borel graphs, and we recommend the reader keep this case in mind while attempting to gain some intuition on the hypotheses in the statements of this section.

\subsection{Orientation with out-degree at most 1} In this subsection, we focus on graphs where a Borel orientation, or a partial, yet rather large, Borel orientation, exists.

\begin{proposition} \label{prop:oriented}
Let $\cG \coloneqq (X, R)$ be a Borel graph with $\deg(\cG) < \infty$. Suppose that there exists a Borel orientation $\vv{R} \subseteq R$ such that $\Degout(X, \vv{R}) \le 1$. Then $\asdim (X,\rho_\cG) \le 1$.
\end{proposition}

\begin{proof}
Since $\vv{R}$ has out-degree at most one, the following bounded-to-one Borel function $f : X \to X$ is well-defined:
\begin{equation*}
f(x) \coloneqq 
\begin{cases}
    y & \text{if} \ x\vv{R}y, \\
    x & \text{otherwise}.
\end{cases}
\end{equation*}
Since $\vv{R} \cup \vv{R}^{-1} = R$, we see that the Borel set
\begin{equation*}
R_f \coloneqq \{(x, y) \in X^2 : x \ne y \text{ and either } f(x) = y \ \text{or} \ f(y) = x\}
\end{equation*}
is equal to $R$. This means that the Borel graph $\cG_f \coloneqq (X, R_f)$ is equal to $\cG$, hence $\asdim (X,\rho_\cG) \le 1$ by \cite[Corollary 8.3]{CJMST-D}.
\end{proof}

\begin{lemma}\label{lemma:delete_edges}
Let $\cG \coloneqq (X, R)$ be a Borel graph with $\deg(\cG) < \infty$. Suppose that for each $r > 0$ there is a symmetric Borel subset $Q \subseteq R$ such that
\begin{enumerate}
    \item \label{item1:delete_edges} $\rho_\cG(q_0, q_1) > r$ whenever $q_0, q_1 \in Q$ represent distinct edges\footnote{That is, if $q_0 = (x, y)$ then $q_1 \ne (x, y)$ and $q_1 \ne (y, x)$; see \eqref{eq:edge_distance} for the definition of $\rho_\cG(q_0, q_1)$.},
    \item \label{item2:delete_edges} the graph $(X, R \setminus Q)$ admits a Borel orientation with out-degree at most 1.
\end{enumerate}
Then $\asdim (X,\rho_\cG) \le 3$.
\end{lemma}

\begin{proof}
Fix $r > 0$. We aim to build a $\rho_\cG$-uniformly bounded Borel equivalence relation $E \subseteq E_\cG$ such that $B_{\rho_\cG}(x, r)$ intersects at most four $E$-classes, for every $x \in X$. To do so, let $Q \subseteq R$ be satisfying the assumptions of the lemma for $2r$, that is
\begin{equation} \label{eq:2r}
\text{  $\rho_\cG(q_0, q_1) >2 r$ whenever $q_0, q_1 \in Q$ are distinct. } \tag{1'}
\end{equation}

Consider $\cG_0 \coloneqq (X, R \setminus Q)$. By \eqref{item2:delete_edges} and Proposition \ref{prop:oriented} we have $\asdim (X,\rho_{\cG_0}) \le 1$, which means that there is a $\rho_{\cG_0}$-uniformly bounded (and thus $\rho_\cG$-uniformly bounded) Borel equivalence relation $E$ such that for every $x \in X$ the ball $B_{\rho_{\cG_0}}(x, 2r)$ intersects at most two distinct $E$-classes. We claim that, for every $x \in X$, the ball $B_{\rho_\cG}(x, r)$ intersects at most 4 $E$-classes.

To see this, let $Y \subseteq X$ be the set of all vertices that are adjacent to some edge in $Q$. Fix $x \in X$. If $B_{\rho_{\cG}}(x, r) \cap Y = \emptyset$, then 
\begin{equation*}
B_{\rho_{\cG}}(x, r) = B_{\rho_{\cG_0}}(x, r) \subseteq B_{\rho_{\cG_0}}(x, 2r)
\end{equation*}
and thus $B_{\rho_{\cG}}(x, r)$ intersects at most two $E$-classes.

Suppose next that $B_{\rho_{\cG}}(x, r) \cap Y \ne \emptyset$. Then by \eqref{eq:2r} there is  (up to symmetry) a unique edge $ (y, z) \in Q$ such that at least one among $y$ and $z$ belongs to  $B_{\rho_{\cG}}(x, r)$. Without loss of generality, we may assume $y \in B_{\rho_{\cG}}(x, r)$. In this case, again by \eqref{eq:2r}, we have 
\[
B_{\rho_{\cG}}(x, r) \subseteq B_{\rho_{\cG}}(y, 2r) \subseteq B_{\rho_{\cG_0}}(y, 2r) \cup B_{\rho_{\cG_0}}(z, 2r).
\]
We deduce that $B_{\rho_{\cG}}(x, r)$ intersects at most 4 different $E$-classes.
\end{proof}

\subsection{Acyclic graphs with degree at most 2}
If $\cG \coloneqq (X, R)$ is an acyclic Borel graph with $\deg(\cG) \le 2$, then it is always possible to remove a very sparse set of edges so that the resulting graph has connected components whose size is uniformly bounded by a given constant. In Lemma \ref{lemma:choose_edges} we prove that this can be done while also ensuring that the set of removed edges is Borel and far away from leaves.

\begin{remark} \label{remark:components}
 Let $\cG\coloneqq (X, R)$ be an acyclic Borel graph with $\deg(\cG) \le 2$. Then each non-trivial connected component $[x]_\cG$ of $\cG$ is isomorphic to one of the following graphs, depending on the number of leaves it has:
\begin{enumerate}[label=(\roman*)]
\item \label{item:Z} if $|\leaves([x]_\cG)| = 0$ then $[x]_\cG \cong (\bbZ, S_\bbZ \cup S_\bbZ^{-1})$, where $S_\bbZ \coloneqq \{(k, k+1) : k \in \bbZ\}$,
\item \label{item:N}  if $|\leaves([x]_\cG)| = 1$ then $[x]_\cG \cong (\bbN, S_\bbN \cup S_\bbN^{-1})$, where $S_\bbN \coloneqq \{(k, k+1) : k \in \bbN\}$.
\item \label{item:n} if $|\leaves([x]_\cG)| = 2$ then there is $n \in \bbN$ such that $[x]_\cG \cong (\{1 ,\dots, n \}, S_n \cup S_n^{-1})$, where $S_n \coloneqq \{ (k, k+1) : k < n \}$.
\end{enumerate}
Set $X_i \coloneqq \{ x \in X : |\leaves([x]_\cG)| = i \}$ for $i = 0, 1, 2$. We then have $X = X_0 \sqcup X_1 \sqcup X_2$ and each $X_i$ is Borel. Indeed, $X_2$ is Borel, being the set of vertices whose connected component is finite (see Remark \ref{remark:degree_maps}), $X_1$ is Borel since it is the set of vertices with infinite connected component and finite distance from $\leaves(\cG)$, and $X_0$ is the complement of $X_1 \sqcup X_2$.

We will implicitly use these facts in the next proof.
\end{remark}

\begin{lemma}\label{lemma:choose_edges}
Let $\cG \coloneqq (X, R)$ be an acyclic Borel graph with $\deg(\cG) \le 2$ and fix $r > 0$. There is a symmetric Borel set $Q \subseteq R$ such that
\begin{enumerate}
    \item \label{item1:choose_edges}  $\rho_\cG(q_0, q_1) > r$ whenever $q_0, q_1 \in Q$ represent distinct edges\footnote{That is, if $q_0 = (x, y)$ then $q_1 \ne (x, y)$ and $q_1 \ne (y, x)$; see \eqref{eq:edge_distance} for the definition of $\rho_\cG(q_0, q_1)$.},
        \item \label{item2:choose_edges} $\rho_\cG(x,y) \ge r$ whenever $x \in X$ is adjacent to an edge in $Q$, and $y \in \leaves(\cG)$,
    \item \label{item3:choose_edges} $\abs{[x]_{\cG_0}} \le 2r +4$, with $\cG_0 \coloneqq (X, R \setminus Q)$, for any $x \in X$.
\end{enumerate}
\end{lemma}

\begin{proof}
We can split the proof in three cases, following the decomposition $X = X_0 \sqcup X_1 \sqcup X_2$ as in Remark \ref{remark:components}. Let $R_i \coloneqq R|_{X_i}$ for $i = 0,1,2$.

For $X_0$, namely the set of vertices whose connected component is infinite, consider the Borel graph $(R_0, S)$ where, given two distinct $q_0, q_1 \in R_0$, we set $q_0 S q_1$ if $\rho_\cG(q_0, q_1) \le r$. The graph $(R_0, S)$ has finite degree since $\cG$ does, hence by \cite[Propositions 4.2 and 4.3]{KST} there exists a maximal \emph{$S$-independent} Borel set $Q_0 \subseteq R_0$, that is $q_0 \neg S q_1$ for all $q_0, q_1 \in Q_0$ and every $q \in R_0$ is $S$-adjacent to some element in $ Q_0$. It is then follows that $Q_0 \cup Q_0^{-1}$ has the desired properties for $(X_0, R_0)$.

Next, in $X_1$, let $Q_1$ be the set of all edges $(x, y) \in R_1$ such that, for some non-zero $k \in \bbN$, one has
\begin{equation*}
\rho_\cG(x, \leaves(\cG)) = k(r+2),\  \rho_\cG(y, \leaves(\cG)) = k(r+2) + 1.
\end{equation*}
Then $Q_1 \cup Q_1^{-1}$ satisfies the required conditions for $(X_1, R_1)$.

Finally, in $X_2$, since the equivalence relation induced by $(X_2, R_2)$ is finite and since $\leaves(\cG)$ is a Borel set, using a Borel selector (see \cite[Theorem 12.16 and Exercise 18.14]{Kechris:CDST}) it is possible to find a Borel set $A \subseteq X_2 \cap \leaves(\cG)$ such that $A$ contains exactly one leaf from each connected component of $(X_2, R_2)$. One can then define $Q_2$ as the set those edges $(x, y) \in R_2$ such that, for some non-zero $k \in \bbN$,
\begin{equation*}
\rho_\cG(x,A) = k(r+2),\  \rho_\cG(y, A) = k(r+2) + 1, \  \rho_\cG(y, \leaves(\cG) \setminus A) \ge r.
\end{equation*}
Then $Q_2 \cup Q_2^{-1}$ is as required for $(X_2, R_2)$.
\end{proof}

\subsection{Main criterion}
\begin{lemma} \label{lemma:quasi-oriented}
Let $\cG \coloneqq (X, R)$ be a Borel graph with $\deg(\cG) < \infty$. Suppose that for every $r > 0$ there exist two Borel  graphs $\cG_0 \coloneqq (X, R_0)$ and $\cG_1 \coloneqq (X, R_1)$ such that
\begin{enumerate}[label=(C.\arabic*)]
\item \label{quasi-oriented:item1} $R = R_0 \sqcup R_1$,
\item \label{quasi-oriented:item2} $\cG_0$ admits a Borel orientation $\vv{R}_0 \subseteq R_0$ with out-degree at most 1,
\item \label{quasi-oriented:item3} $\deg_{\cG_1}(x) + \degoutz(x) \le 2$ for all $x \in X$,
\item \label{quasi-oriented:item4} $\cG_1$ is acyclic,
\item \label{quasi-oriented:item5} if $1 < |[x]_{\cG_1}| < \infty$ and $\degoutz(y) = 1$ for every $y \in \leaves(\cG_1 \cap [x]_{\cG_1})$, then $ |[x]_{\cG_1}| > r$.
\end{enumerate}
Then $\asdim(X, \rho_\cG) \le 3$.
\end{lemma}

\begin{proof}
We aim to apply Lemma \ref{lemma:delete_edges} to $\cG$. In order to do so, fix $r > 0$ and find Borel graphs  $\cG_0 \coloneqq (X, R_0)$ and $\cG_1 \coloneqq (X, R_1)$ satisfying the assumptions of the present lemma for the value $2r+4$. In particular, $\cG_0$ and $\cG_1$ satisfy \ref{quasi-oriented:item1}-\ref{quasi-oriented:item4} and
\begin{equation} \label{eq:5a}
 |[x]_{\cG_1}| > 2r +4, \text{ if } 1 < |[x]_{\cG_1}| < \infty \text{ and } \degoutz(y) = 1, \, \forall y \in \leaves(\cG_1 \cap [x]_{\cG_1}). \tag{C.5'}
 \end{equation}

Condition \ref{quasi-oriented:item3} entails $\deg(\cG_1) \le 2$, hence we can apply Lemma \ref{lemma:choose_edges} to find a symmetric Borel subset $Q \subseteq R_1$ such that
\begin{enumerate}[label=(\roman*)]
    \item \label{item1:reflemma} $\rho_\cG(q_0, q_1) > r$ whenever $q_0, q_1 \in Q$ represent distinct edges,
    \item \label{item2:reflemma} $\rho_{\cG_1}(x,y) \ge r$ whenever $x \in X$ is $R_1$-adjacent to an edge in $Q$, and $y \in \leaves(\cG_1)$,
    \item \label{item3:reflemma} $\abs{[x]_{(X, R_1 \setminus Q)}} \le 2r+4$ for any $x \in X$.
\end{enumerate}

Define $\cG_2 \coloneqq( X, R_1 \setminus Q)$. We show now how to construct a Borel orientation $\vv{R}_2$ on $\cG_2$ such that $\vv{R}_0 \cup \vv{R}_2$ is a Borel orientation of $(X, R \setminus Q)$ with out-degree at most 1. In order to define such $\vv{R}_2$, consider the Borel sets (see Remark \ref{remark:degree_maps})
\begin{align*}
Y_0 &\coloneqq \left\{ x \in X : |[x]_{\cG_2} | > 1 \text{ and there is } y \in [x]_{\cG_2} \text{ such that}\, \degoutz(y) =1 \right\}, \\
Y_1 &\coloneqq \left\{ x \in X : |[x]_{\cG_2} | > 1 \right\} \setminus Y_0.
\end{align*}

Condition \ref{quasi-oriented:item3} implies that if $x \in Y_0$ and $y \in [x]_{\cG_1}$ has $\degoutz(y) = 1$, then $y\in \leaves(\cG_1)$. If there were two distinct leaves $y_0,y_1\in [x]_{\cG_1}$ with $\degoutz(y_i) = 1$, for $i = 0,1$, then by \eqref{eq:5a} we would get $|[x]_{\cG_1}| > 2r+4$. This fact, combined with  \ref{item3:reflemma}, is sufficient to deduce that, for every $x \in X$, there exists at most one vertex $y \in [x]_{\cG_2}$ with $\degoutz(y) = 1$. In other words, if $A_0 \coloneqq (\degoutz)^{-1}(1) \cap Y_0$, then $|A_0 \cap [x]_{\cG_2}| = 1$ for every $x \in Y_0$.

On other hand, since $E_{\cG_2}$ is a finite equivalence relation by \ref{item3:reflemma}, and since finite Borel equivlence relations admit Borel selectors whose image is Borel (see \cite[Theorem 12.16 and Exercise 18.14]{Kechris:CDST}), there exists a Borel set $A_1 \subseteq Y_1$ such that $|A_1 \cap [x]_{\cG_2}| = 1$ for every $x\in Y_1$.

Set $A \coloneqq A_0 \cup A_1$. Since $\deg(\cG_2) \le 2$, the following is a well-defined orientation on $\cG_2$
\begin{equation*}
(x,y) \in \vv{R}_2 \Leftrightarrow \rho_{\cG_2}(y, A) < \rho_{\cG_2}(x, A).
\end{equation*}
By construction, $\vv{R}_0 \cup \vv{R}_2$ is a Borel orientation of $(X, R \setminus Q)$ with out-degree at most 1.

In order to apply Lemma \ref{lemma:delete_edges} and conclude the proof, it suffices to show that if $(x_0, y_0), (x_1, y_1) \in Q$ are such that $x_0 \ne x_1$ and $x_0 \ne y_1$, then $\rho_\cG(x_0, x_1) > r$. Without loss of generality, we may assume that the shortest path in $\cG$ connecting $x_0$ to $x_1$ only has edges in $R \setminus Q$ (note that the shortest path connecting $x_0$ to $x_1$ cannot be composed exclusively of edges in $Q$ by \ref{item1:reflemma}, unless $r < 1$, in which case $\rho_\cG(x_0, x_1) > r$ is automatic). Arguing as in  Proposition \ref{prop:oriented}, one sees that the Borel function $f \colon X \to X$
\begin{equation*}
f(x) \coloneqq
\begin{cases}
    y & \text{if } (x,y) \in \vv{R}_0 \sqcup \vv{R}_2, \\
    x & \text{otherwise},
\end{cases}
\end{equation*}
is well-defined. Since $x_0, x_1$, as well as the shortest path connecting them, are in the same connected component of $(X,R \setminus Q)$, and since $\vv{R}_0 \sqcup \vv{R}_2$ has out-degree at most 1, it is possible to prove by induction on $\rho_\cG(x_0,x_1)$ that
\begin{equation} \label{eq:geodesic}
\rho_\cG(x_0, x_1) = \min \{ n+m : f^n(x_0) = f^m(x_1) \}.
\end{equation}

Note that, by \ref{item2:reflemma}, $x_0$ is not a leaf of $\cG_1$,  hence by \ref{quasi-oriented:item3} we deduce $\degoutz(x_0) = 0$. This implies in particular that either $x_0 = f(x_0)$ or $x_0 R_1 f(x_0)$, and thus $f(x_0) \in [x_0]_{\cG_1}$. The same argument can be inductively repeated to infer that
\begin{equation*}
f^k(x_0) \in [x_0]_{\cG_1} \text { and either } f^k(x_0) = f^{k+1}(x_0) \text{ or } f^k(x_0) R_1 f^{k+1}(x_0), \quad k= 0,\dots,  \lceil r \rceil .
\end{equation*}
The same statement can be proved verbatim for $x_1$.

Thus, if $[x_0]_{\cG_1} \ne [x_1]_{\cG_1}$, then $f^k(x_0) \ne f^j(x_1)$ for all $0 \le k, j \le r$. If, on the other hand,  $[x_0]_{\cG_1} = [x_1]_{\cG_1}$, then condition \ref{item1:reflemma} implies $f^k(x_0) \ne f^j(x_1)$ whenever $k+j \le r$. In either case, using \eqref{eq:geodesic}, we conclude that $\rho_\cG(x_0, x_1) > r$.
\end{proof}

\section{U-amenability and finite Borel asymptotic dimension} \label{section:asdim}
In this section, we prove the main technical result of the paper, Theorem \ref{thm:ua_asdim}, asserting that a u-amenable acyclic Borel graph with finite degree has finite Borel asymptotic dimension, and thus induces a hyperfinite equivalence relation.

Before proceeding with the proof, we introduce some useful notation. Given a graph $\cG \coloneqq (X, R)$ and $xRy$, define the graph
\begin{equation} \label{eq:Gxy}
\cG_{x,y} \coloneqq (X, R \setminus \{ (x,y), (y,x) \}),
\end{equation}
which is the graph obtained from $\cG$ after removing the edge $xRy$. Note that, in case $\cG$ is acyclic, then
\begin{equation*}
[x]_{\cG_{x,y}} = \{z \in [x]_\cG \colon \rho_\cG(z, x) < \rho_\cG(z, y)\}, \, [y]_{\cG_{x,y}} = \{z \in [x]_\cG \colon \rho_\cG(z, y) < \rho_\cG(z,x)\},
\end{equation*}
and moreover
\begin{equation} \label{eq:partition}
[x]_\cG = [y]_\cG = [x]_{\cG_{x,y}} \sqcup [y]_{\cG_{x,y}}.
\end{equation}
Finally, for $\lambda \in \ell^1([x]_\cG)$ and $A \subseteq [x]_\cG$, we write $\lambda(A)$ for the sum $\sum_{x \in A} \lambda(x)$.

\begin{lemma}\label{lemma:theta_maps}
Let $\cG \coloneqq (X, R)$ be an acyclic Borel graph and let $\lambda : E_\cG \to \mathbb{R}_{\ge 0}$ be a Borel function such that $\lambda(x, \cdot)$, which we abbreviate as $\lambda_x$, belongs to $\ell^1([x]_\cG)$ for all $x \in X$. Then the maps
\[
\begin{aligned}[c]
\theta_0 \colon R &\to \bbR \\
(x,y) &\mapsto \lambda_{x}([x]_{\cG_{x,y}})
\end{aligned}
\qquad
\begin{aligned}[c]
\theta_1 \colon R &\to \bbR \\
(x,y) &\mapsto \lambda_{y}([y]_{\cG_{x,y}})
\end{aligned}
\]
are Borel.
\end{lemma}

\begin{proof}
It sufficies to prove the statement for $\theta_0$, since $\theta_1(x,y) =\theta_0(y,x)$. By the Lusin--Novikov uniformization theorem (\cite[Theorem 18.10]{Kechris:CDST}) there are Borel functions $F_n \colon X \to X$, for $n \in \bbN$, such that $E_\cG = \{(x,y) \in X^2 : \exists n \, (y = F_n(x)) \}$.
We inductively define a sequence of Borel functions $w_n \colon R \to \bbR$ as follows:
\begin{align*}
w_0(x,y) & \coloneqq 0 \text{ for all }  xRy\\
w_{n+1}(x,y) &\coloneqq \begin{cases} w_n(x,y) + \lambda_x(F_{n+1}(x)) & \smash{\raisebox{1.2ex}{if $F_{n+1}(x) \ne F_j(x)$ for all $j \le n$,}} \\
&\smash{\raisebox{1.5ex}{$\rho_\cG(x, F_{n+1}(x)) < \rho_\cG(y, F_{n+1}(x))$}} \\
w_n(x,y) &\text{otherwise.}
\end{cases}
\end{align*}
Since $\cG$ is acyclic, the sequence $(w_n)_{n \in \bbN}$ pointwise converges to $\theta_0$, which is therefore Borel.
\end{proof}

\begin{theorem}\label{thm:ua_asdim}
Let $\cG \coloneqq(X, R)$ be an acyclic Borel graph with $\deg(\cG) < \infty$. If $\cG$ is u-amenable then $\asdim(X, \rho_\cG) \le 1$. In particular, $E_\cG$ is hyperfinite.
\end{theorem}

\begin{proof}
We aim to apply Lemma \ref{lemma:quasi-oriented} to $\cG$. To do so, fix $r > 0$. By u-amenability there exists a 
Borel map
\[
\lambda \colon E_\cG \to \bbR_{\ge 0}
\]
such that $\lambda_x \in \ell^1([x]_{\cG})$ with $\| \lambda_x \|_1 = 1$ for all $x \in X$, and such that
\begin{equation} \label{eq:amenable}
\| \lambda_x - \lambda_y \|_1 < \frac{1}{12}, \quad \text{if } \rho_\cG(x, y) \le r +2, \, x,y \in X.
\end{equation}

Let $\theta_0, \theta_1 \colon R \to \bbR_{\ge 0}$ be the maps defined as
\begin{equation*}
\theta_0(x, y) \coloneqq \lambda_x([x]_{\cG_{x,y}}) \text{ and } \theta_1(x, y) \coloneqq \lambda_y([y]_{\cG_{x,y}}), \quad (x,y) \in R,
\end{equation*}
which are Borel by Lemma \ref{lemma:theta_maps} (see \eqref{eq:Gxy} for the definition of $\cG_{x,y}$).

Consider the set of edges
\begin{equation*}
R_0 \coloneqq \left\{ (x,y) \in R \colon  \theta_0(x, y), \theta_1(x, y) \notin  [5/12,7/12 ] \right\}.
\end{equation*}
The set $R_0$ is Borel since $\theta_0$ and $\theta_1$ are. Given $xRy$ we have 
\begin{equation*}
\abs{(1 - \theta_0(x, y)) - \theta_1(x, y)} \stackrel{\eqref{eq:partition}}{=} \abs{\lambda_x([y]_{\cG_{x,y}}) - \lambda_y([y]_{\cG_{x,y}})} \stackrel{\eqref{eq:amenable}}{<} 1/12.
\end{equation*}
Hence, if $xR_0y$, then exactly one between either
\begin{equation} \label{eq:1}
\theta_0(x, y) > 7/12\ \text{ and }\ \theta_1(x, y) < 5/12
\end{equation}
or
\begin{equation} \label{eq:2}
\theta_0(x, y) < 5/12 \ \text{ and }\  \theta_1(x, y) > 7/12
\end{equation}
holds. Using \eqref{eq:1}-\eqref{eq:2}, along with $\theta_0(x,y) = \theta_1(y,x)$, we define a Borel orientation $\vv{R}_0 \subset R_0$ by setting
\begin{equation*}
\vv{R}_0 \coloneqq \{ (x,y) \in R_0 : \theta_0(x, y) < 5/12, \, \theta_1(x, y) > 7/12 \}.
\end{equation*}

Let $R_1 \coloneqq R \setminus R_0$ and $\cG_1 \coloneqq (X, R_1)$. We verify now that $\cG_0 = (X, R_0)$ with orientation $\vv{R}_0$ and $\cG_1$ satisfy the five conditions of Lemma \ref{lemma:quasi-oriented}, which is enough to conclude the proof, since $\asdim(X, \rho_\cG) \le 3$ implies that $\asdim(X, \rho_\cG) \le 1$, by \cite[Theorem 1.1]{CJMST-D}.

\textbf{\ref{quasi-oriented:item1}:} this condition holds by definition.

\textbf{\ref{quasi-oriented:item2}:} let $x \in X$ and suppose that $y_0, y_1 \in X$ are distinct such that $x \vv{R}_0 y_0$ and $x \vv{R}_0 y_1$. As $\cG$ is acyclic, the trees $[y_0]_{\cG_{x, y_0}}$
and $[y_1]_{\cG_{x, y_1}}$ are disjoint subsets of $[x]_\cG$, and therefore we have
\begin{equation*}
\begin{split}
1 &\ge \lambda_x([y_0]_{\cG_{x, y_0}}) + \lambda_x([y_1]_{\cG_{x, y_1}})\\
&\stackrel{\mathclap{\eqref{eq:amenable}}}{>} \lambda_{y_0}([y_0]_{\cG_{x, y_0}}) + \lambda_{y_1}([y_1]_{\cG_{x, y_1}}) - 2/12 \\
& > 7/12+7/12-2/12 = 1,
\end{split}
\end{equation*}
which is a contradiction.

\textbf{\ref{quasi-oriented:item3}:} suppose that $y_0, y_1$ and $y_2$ are three distinct points in $[x]_\cG$ such that either $x \vv{R}_0 y_i$ or $x R_1 y_i$ holds for all $i=0, 1, 2$. Then $[y_0]_{\cG_{x, y_0}}, [y_1]_{\cG_{x, y_1}}$ and $[y_2]_{\cG_{x, y_2}}$ are disjoint and we have
\begin{align*}
1 &\ge \lambda_x([y_0]_{\cG_{x, y_0}}) + \lambda_x([y_1]_{\cG_{x, y_1}}) + \lambda_x([y_2]_{\cG_{x, y_2}})\\
&\stackrel{\mathclap{\eqref{eq:amenable}}}{>} \lambda_{y_0}([y_0]_{\cG_{x, y_0}}) + \lambda_{y_1}([y_1]_{\cG_{x, y_1}}) + \lambda_{y_2}([y_2]_{\cG_{x, y_2}}) - 3/12 \\
& \ge 5/12+5/12+5/12 - 3/12 = 1.
\end{align*}

\textbf{\ref{quasi-oriented:item4}:} this condition is automatic since $\cG$ is acyclic.

\textbf{\ref{quasi-oriented:item5}:} assume that there are two leaves $x_0, x_1$ in the same connected component $[x]_{\cG_1}$ of $\cG_1$ such that there are $y_0, y_1 \in X$ with $x_0 \vv{R}_0 y_0$ and $x_1 \vv{R}_0 y_1$. Note that $[y_0]_{\cG_{x_0, y_0}}$ and $[y_1]_{\cG_{x_1, y_1}}$ are disjoint. If $\rho_{\cG_1}(x_0, x_1) \le r$ then $\rho_{\cG_1}(y_0, y_1) \le r+2$ hence
\begin{align*}
1 &\ge \lambda_{y_0}([y_0]_{\cG_{x_0, y_0}}) + \lambda_{y_0}([y_1]_{\cG_{x_1, y_1}})\\
&\stackrel{\mathclap{\eqref{eq:amenable}}}{>} \lambda_{y_0}([y_0]_{\cG_{x_0, y_0}}) + \lambda_{y_1}([y_1]_{\cG_{x_1, y_1}}) - 1/12 \\
& > 7/12+7/12-1/12 > 1.
\end{align*}
This is a contradiction, hence $\rho_{\cG_1}(x_0, x_1) > r$ which in turn gives $|[x]_{\cG_1}| > r$.

Hyperfiniteness of $E_\cG$ follows by \cite[Theorem 1.7]{CJMST-D}.
\end{proof}

We put all pieces finally together and deduce Theorem \ref{thm_main:hyper_treeable}, and its corollaries Theorem \ref{thm_main:Fk}, Corollary \ref{cor_intro:group}, and Corollary \ref{cor_intro:bounded}.
\begin{corollary}\label{cor: main result}
Let $E$ be a countable Borel equivalence relation which is treeable and hyper-u-amenable. Then $E$ is hyperfinite.
\end{corollary}

\begin{proof}
Let $\cG$ be an acyclic Borel graph such that $E_\cG = E$. By Proposition \ref{prop: increasing union graphings} the graph $\cG$ is a union of an increasing sequence of u-amenable subgraphs $\cG_n$ that have finite degree. The graphs $\cG_n$ are acyclic, and therefore $\asdim(X, \rho_{\cG_n}) < \infty$ for all $n \in \bbN$, by Theorem \ref{thm:ua_asdim}. The conclusion then follows by Proposition \ref{prop:union}.
\end{proof}

\begin{corollary} \label{cor:applications}
Let $X$ be a standard Borel space and let $E \subseteq X^2$ be a treeable countable Borel equivalence relation. Suppose that either one of the following three assumptions holds:
\begin{enumerate}
\item $E$ is measure-hyperfinite and $E = E_G^X$ for an action $G \curvearrowright X$ by a countable group such that the stabilizer of every point is amenable, and there is a Polish $\sigma$-compact topology on $X$ for which such action is continuous,
\item $E = E_G^X$ for some Borel action $G \curvearrowright X$ by a countable amenable group $G$,
\item $E$ is amenable and Borel bounded.
\end{enumerate}
Then $E$ is hyperfinite.
\end{corollary}
\begin{proof}
The equivalence relation $E$ is hyper-u-amenable under either of the three assumptions in the statement, respectively by Corollary \ref{cor: continuous action}, Proposition \ref{prop:group_action}, and Proposition \ref{prop: Borel bounded}. The conclusion then follows by Corollary \ref{cor: main result}.
\end{proof}

\bibliography{main.bib}
\end{document}